\documentclass{amsart}

\usepackage[all,poly]{xy}
\usepackage{amsfonts,amsmath,amssymb}

\theoremstyle{plain}
\newtheorem{thm}{Theorem}[section]
\newtheorem{prop}[thm]{Proposition}
\newtheorem{lem}[thm]{Lemma}
\newtheorem{cor}[thm]{Corollary}

\newtheorem{conj}{Conjecture}
\newtheorem{defn}[thm]{Definition}

\newtheorem{exmp}[thm]{Example}

\begin{document}

\title{Numerical Measures for Two-Graphs}

\author{David M. Duncan}
\address{Coastal Carolina University, Conway, SC}
\email{dduncan@coastal.edu}
\author{Thomas R. Hoffman}
\email{thoffman@coastal.edu}
\author{James P. Solazzo}
\email{jsolazzo@coastal.edu}


\begin{abstract}
We study characteristics which might distinguish two-graphs by introducing
different numerical measures on the collection of graphs on $n$ vertices. Two conjectures are stated, one using these numerical measures and the other
using the deck of a graph, which suggest that there is a finite set of conditions differentiating two-graphs.
We verify that, among the four non-trivial non-isomorphic regular two-graphs on
$26$ vertices, both conjectures hold.
\end{abstract}

\maketitle

\section{Introduction}

The notion of a frame was introduced over 50 years ago in the work
of Duffin and Schaeffer \cite{KC}. However, in the last few years
frames have caught the attention of mathematicians from a variety of
disciplines. This is due in large part to the fact that until
recently little was known about frames. Furthermore, frame theory
has been shown to have a number of practical applications
encompassing quantization, signal reconstruction and coding theory,
to name a few. In \cite{BP} and \cite{HP}, Bodman, Holmes, and
Paulsen use frame theory to answer certain questions related to the
``lost package problem'', a problem in engineering. In these papers,
they show that a certain family of frames, two-uniform frames, are
optimal for one and two erasures. Strohmer and Heath, in \cite{SH},
prove similar results about two-uniform frames and make a clear
connection between finite frames and areas such as spherical codes,
equidistant point sets, two-graphs,
 and sphere packings.
The aforementioned three papers take advantage of known results and
classical constructions of two-graphs to prove several results about
two-uniform frames. Much of the theory on two-graphs can be found in
the works of J.J. Seidel, e.g., \cite{BMS}, \cite{CM}, \cite{Sei},
and \cite{vLS}. In fact, two-uniform frames are in one-to-one
correspondence with regular two-graphs. It is this relationship
between frame theory and graph theory which motivates the work in
this paper.

A two-graph is the collection of graphs obtained from a graph $X$ on
$n$ vertices by {\it switching} on every subset of the vertex set of
$X$. For this reason a two-graph is sometimes referred to as a {\it
switching class}. Given graphs $X$ and $Y$, if any pair of
representatives from their respective switching classes are
isomorphic, we say that the graphs $X$ and $Y$ are {\it switching
equivalent}. We will make the above terminology precise in Section
\ref{Sec:Switching}. In this paper we explore characteristics which
might differentiate the switching equivalent classes for graphs on
$n$ vertices. In the language of Seidel \cite{Sei}, we investigate
the isomorphism classes of two-graphs. This exploration led us to
conjecture that graphs, $X_1$ and $X_2$, are switching equivalent if
and only if a finite set of norm conditions are satisfied. This is
Conjecture \ref{oneequivalent} in Section \ref{Sec:OneNorms}. We
verify this conjecture for $n$ less than or equal to $10$ and also
for four ``special'' graphs on $26$ vertices. We considered these
four graphs ``special'' since they mark the first appearance of
nontrivial non-isomorphic regular two-graphs for fixed $n$.

This paper is organized as follows. In Section \ref{Sec:Switching}
we fix notation and emphasize the distinction between the switching class
and switching equivalent class of a graph. The
{\it infinity norm} of a graph is defined in Section \ref{Sec:InfNorms} and is shown not sufficient to determine the switching equivalent class of a graph.
The spectra of switching equivalent classes are considered in Section \ref{Sec:DS}.
Examples of cospectral classes are known, but we include
this section for completeness of the discussion. Section
\ref{Sec:Decks} introduces another new conjecture for switching
equivalent classes in terms of decks of graphs. In Section \ref{Sec:OneNorms}, the {\it $1$-norm} of a graph is defined along with a corresponding conjecture for switching equivalent classes. Section \ref{Sec:26} verifies the conjectures of Sections \ref{Sec:Decks} and \ref{Sec:OneNorms} on the smallest nontrivial example of non-isomorphic two-graphs on $n$ vertices. Lastly, Appendix \ref{App:Frames} gives an introduction to
frame theory and provides motivation for our definitions of the
{\it infinity norm} and the {\it $1$-norm} of a graph.

\section{Switching}\label{Sec:Switching}

All graphs considered in this paper will be simple, i.e. undirected,
without loops, without multiple edges, and finite. Denote by $A(X)$,
$V(X)$ and $E(X)$ the adjacency matrix, the set of vertices and the
set of edges of the graph $X$, respectively. We also use $I_n$ for the $n
\times n$ identity matrix and $J_n$ for the $n \times n$ matrix of
all ones.

\begin{defn}
Given a graph $X$ on $n$ vertices, the {\bf Seidel adjacency matrix}
of $X$ is defined to be the $n\times n$ matrix $S(X):=(s_{ij})$
where $s_{i,j}$ is defined to be $-1$ when $i$ and $j$ are adjacent
vertices, $+1$ when $i$ and $j$ are not adjacent, and 0 when $i=j$.
\end{defn}

The Seidel adjacency matrix of $X$ is related to the usual adjacency
matrix $A(X)$ by
$$
S(X)=J_n-I_n-2A(X).
$$

\begin{defn}
Let $X$ be a graph and $\tau \subseteq V(X)$.
Now define the  graph $X^{\tau}$ to be the graph that arises from
$X$ by changing all of the edges between $\tau$ and $V(X)- \tau$ to
nonedges and all the nonedges between $\tau$ and $V(X)-\tau$ to
edges. This operation is called {\bf switching} on the subset $\tau$, see \cite{GR}.
\end{defn}

The operation of switching is an equivalence relation on the
collection of graphs on $n$ vertices. This can be seen by observing
that if $\tau \subseteq V(X)$, then switching on $\tau$ is
equivalent to conjugating $S(X)$ by the diagonal matrix $D$ with
$D_{ii}=-1$ when $i \in \tau$ and $1$ otherwise. The {\bf switching
class} of $X$, denoted $[X]$, is the collection of graphs that can
be obtained from $X$ by switching on every subset of $V(X)$. A
switching class of graphs is also known as a {\bf two-graph}.

\begin{defn}
The graphs $X$ and $Y$ on $n$ vertices are called {\bf switching
equivalent} if $Y$ is isomorphic to $X^{\tau}$ for some $\tau
\subset V(X)$, see \cite{GR}.
\end{defn}

Switching equivalence is also an equivalence relation on the
collection of graphs on $n$ vertices. The {\bf switching equivalent
class} of $X$, denoted $[[X]]$, is the collection of graphs that can
be obtained from $X$ by conjugating $S(X)$ by a signed permutation
matrix, i.e. the product of a permutation matrix and a diagonal
matrix of $\pm 1's$. Thus, the spectrum of the Seidel adjacency
matrices of switching equivalent graphs are identical. Note that
$[X]$ is a subset of $[[X]]$ for any graph. For the complete graph and empty graph
on $n$ vertices, their switching classes are equal to their switching equivalent classes.
The following examples are similar to Examples 3.6 and 3.7 in
\cite{Sei}.

\begin{exmp}\label{3vertexample} On $3$ vertices there are $4$
non-isomorphic graphs, $2$ distinct switching classes of graphs, and
$2$ distinct switching equivalent classes of graphs. The $4$
non-isomorphic graphs $X_1,$ $X_2$, $X_3,$ and $X_4$ are listed
below.
$$
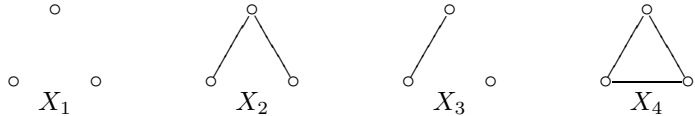
\begin{array}{cccc}

\xy/r.8pc/:,
{\xypolygon3{~>{}~:{(1.875,0):}{\circ}}},
\endxy
\hspace{1cm}
&

\xy/r.8pc/:,
{\xypolygon3{~>{}~:{(1.875,0):}{\circ}}},
"1";"2"**@{-},
"1";"3"**@{-},
\endxy
\hspace{1cm}

&

\xy/r.8pc/:,
{\xypolygon3{~>{}~:{(1.875,0):}{\circ}}},
"1";"2"**@{-},
\endxy
\hspace{1cm}

&

\xy/r.8pc/:,
{\xypolygon3{~>{}~:{(1.875,0):}{\circ}}},
"1";"2"**@{-},
"1";"3"**@{-},
"2";"3"**@{-},
\endxy \\
X_1 \hspace{1cm} & X_2 \hspace{1cm} & X_3 \hspace{1cm} & X_4 \\
\end{array}
$$
\vspace{.5cm}
Thus $[X_1]=[[X_1]]=[X_2]$ and $[X_3]=[[X_3]]=[X_4]$ but $[X_1]\not=[X_3]$.
\end{exmp}

\begin{exmp}\label{4vertexample} On $4$ vertices there are
$11$ non-isomorphic graphs, $8$ distinct switching classes of graphs, and
$3$ distinct switching equivalent classes of graphs. The $11$
non-isomorphic graphs are $X_1,...,X_6$, listed below, and their
complements $X_6,...,X_{11}$.
$$
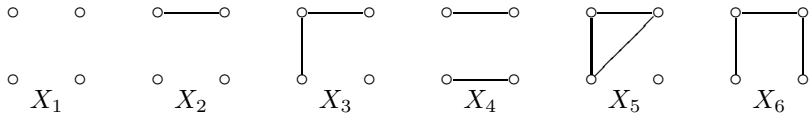
\begin{array}{cccccc}
\xy/r.8pc/:,
{\xypolygon4{~>{}~:{(1.875,0):}{\circ}}},
\endxy

\hspace{.5cm}
&

\xy/r.8pc/:,
{\xypolygon4{~>{}~:{(1.875,0):}{\circ}}},
"1";"2"**@{-},
\endxy
\hspace{.5cm}
&

\xy/r.8pc/:,
{\xypolygon4{~>{}~:{(1.875,0):}{\circ}}},
"1";"2"**@{-},
"2";"3"**@{-},
\endxy
\hspace{.5cm}
&

\xy/r.8pc/:,
{\xypolygon4{~>{}~:{(1.875,0):}{\circ}}},
"1";"2"**@{-},
"3";"4"**@{-},
\endxy
\hspace{.5cm}
&

\xy/r.8pc/:,
{\xypolygon4{~>{}~:{(1.875,0):}{\circ}}},
"1";"2"**@{-},
"2";"3"**@{-},
"3";"1"**@{-},
\endxy
\hspace{.5cm}
&

\xy/r.8pc/:,
{\xypolygon4{~>{}~:{(1.875,0):}{\circ}}},
"1";"2"**@{-},
"2";"3"**@{-},
"1";"4"**@{-},
\endxy \\

X_1 \hspace{.5cm}& X_2 \hspace{.5cm}& X_3 \hspace{.5cm}& X_4 \hspace{.5cm}& X_5 \hspace{.5cm}& X_6 \\
\end{array}
$$
The distinct switching classes are $[X_1]$, $[X_2]$, $[X_4]$, and the remaining graphs with one edge. The distinct switching equivalent classes are $[[X_1]]$, $[[X_2]]$, and $[[X_4]]$.
\end{exmp}

Table \ref{table1} provides partial data on the number of non-isomorphic graphs, switching classes, and switching equivalent classes. In \cite{McK}, McKay gives the number of non-isomorphic graphs on $n$ vertices up to $n=12$. Although, as indicated in Table \ref{table1} there is no known formula for the number of switching equivalent classes, in \cite{Sei}, Seidel states that $\displaystyle{2^{\frac{1}{2}n^2-O(n\log n)}}$ is an asymptotic formula for the number of switching equivalent classes on $n$ vertices due to P.M. Neumann through private communication with Seidel. This formula is only slightly better than the number of switching classes on $n$ vertices. The authors constructed representatives of the switching equivalent classes using the software package GAP, \cite{GAP}, but these numbers are already documented in \cite{BMS}.

\begin{table}[ht]
\begin{center}
\begin{tabular}{|c|c|c|c|}  \hline
$n$ & non-isomophic & switching classes & switching equivalent classes \\ \hline
3 & 4  & 2    &  2 \\ \hline
4 & 11 & 8    &  3 \\ \hline
5 & 34   & 64   & 7  \\ \hline
6 & 156   & 1024 & 16 \\ \hline
7  & 1044    &  32,768    & 54 \\ \hline
8  & 12,346   &   $2^{21}$   & 243 \\ \hline
9  & 274,668   &   $2^{28}$   & 2038 \\ \hline
10 & 12,005,168   &   $2^{36}$   & 33,120 \\ \hline
n  &   no known formula & $\displaystyle{ 2^{\frac{(n-1)(n-2)}{2}}}$ & no known formula \\ \hline
\end{tabular}
\caption{Class Sizes}\label{table1}
\end{center}
\end{table}

Using the terminology found in \cite{BMS}, if every vertex of a
given graph has even degree, we call it an Euler graph. The
following results allow us to use Euler graphs as unique, up to
isomorphism, representatives of switching equivalent classes when
the number of vertices is odd.
\begin{thm}\label{EulerSwitch}
If $G$ is a graph with an odd number of vertices, $G$ is switching equivalent to an Euler graph.
\end{thm}

\begin{proof}
Let $O$ be the set of odd degree vertices of $G$. Let
$\tau=V(G)\setminus O$, the set of vertices of $G$ with even degree.
Take $v$ to be a vertex in $\tau$ and use $\operatorname{N}(v)$
to denote the vertices adjacent to $v$. Since $|O|$ is even,
$|\operatorname{N}(v)\cap O|$ and $|O\setminus
\operatorname{N}(v)|$ have the same parity. So switching $G$ on
the set $\tau$ preserves the parity of the degree of each vertex in
$\tau$. Similarly, let $u$ be a vertex in $O$. Since $|O|$ is even,
$|V(G)\setminus O|$ must be odd. Therefore
$|\operatorname{N}(u)\cap\tau|$ and $|\tau\setminus
\operatorname{N}(u)|$ have different parities and switching
$G$ on the set $\tau$ changes the parity of each vertex in $O$.
Thus, $G^\tau$ has only vertices of even degree and is an Euler
graph.
\end{proof}

\begin{cor}\label{EulerRep}
When $n$ is odd, every switching equivalence class contains exactly one Euler graph, up to isomorphism.
\end{cor}

\begin{proof}
By Theorem 2.2 of \cite{BMS}, there are as many switching equivalent
classes as there are isomorphism classes of Euler graphs. Combining
this with the above theorem gives the desired result.
\end{proof}

\begin{cor}\label{EulerIso}
When $n$ is odd, two Euler graphs are switching equivalent if and only if they are isomorphic.
\end{cor}

While Theorem 2.2 of \cite{BMS} holds for $n$ even as well as odd,
Theorem \ref{EulerSwitch} fails when $n$ is even. The following example demonstrates
this fact.

\begin{exmp}
By Theorem 2.2 of \cite{BMS}, there are three Euler graphs up to isomorphism on four vertices. They are
$$
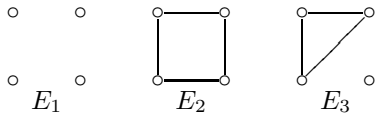
\begin{array}{ccc}
\xy/r.8pc/:,
{\xypolygon4{~>{}~:{(1.875,0):}{\circ}}},
\endxy

\hspace{.5cm}
&

\xy/r.8pc/:,
{\xypolygon4{~>{}~:{(1.875,0):}{\circ}}},
"1";"2"**@{-},
"3";"2"**@{-},
"3";"4"**@{-},
"1";"4"**@{-},
\endxy
\hspace{.5cm}
&

\xy/r.8pc/:,
{\xypolygon4{~>{}~:{(1.875,0):}{\circ}}},
"1";"2"**@{-},
"1";"3"**@{-},
"2";"3"**@{-},
\endxy
\hspace{.5cm}
\\

E_1 \hspace{.5cm}& E_2 \hspace{.5cm}& E_3 \hspace{.5cm} \\

\end{array}
$$
Switching $E_2$ on a pair of nonadjacent vertices will result in an
empty graph, the same class as $E_1$. Switching $E_3$
on the isolated vertex results in the graph $K_4$; so $E_3$ is in
the class of the complete graph. None of these three graphs are in
the class $[[X_2]]$ from Example \ref{4vertexample}.

\end{exmp}

\section{Infinity Norms}\label{Sec:InfNorms}

Our first attempt to find a characteristic which differentiates switching equivalent classes comes in
the form of a matrix norm similar to that used in \cite{BP}.

\begin{defn} \label{inftynormdefA}
Let $\mathcal{D}_m$ denote the set of diagonal matrices that have
exactly $m$ diagonal entries equal to one and $n-m$ entries equal to
zero.  Given a graph $X$ on $n$ vertices, set
$$
e_m^{\infty}(X):=\mathrm{max}\{ \| D(I+cS)D \|  : D \in \mathcal{D}_m \},
$$
where $S$ is the Seidel adjacency matrix of $X$, $c=\frac{1}{n-1}$, and the norm of the matrix is understood to be the operator norm.
\end{defn}

The infinity norm of a graph $X$, $e_m^{\infty}(X)$, from Definition
\ref{inftynormdefA} is the maximum of a set of $n \choose{m}$
numbers which correspond to the collection of induced subgraphs
on $m$ vertices. Moreover, our choice for the constant $c$ is the
smallest $c$ which guarantees $D(I+cS)D$ is a positive semi-definite
matrix. So $\|D(I+cS)D\|$ is the largest eigenvalue of $D(I+cS)D$.
Lemma \ref{largesteigenvalue} and Proposition  \ref{moo} below
verify this statement.

\begin{lem} \label{largesteigenvalue}
If $S$ is a Seidel adjacency matrix for a graph $X$ on $n$ vertices, then $\|S \|$ is at most $n-1$.
\end{lem}

\begin{proof}
First note that the largest eigenvalue of $J_n$, the matrix of all ones, is $n$. For any vector $x$ in $\mathbb{R}^n$ and any $S$, changing signs to make all their entries positive can only increase
the value of the expression
$$
\frac{\langle x, (I_n+S)x \rangle }{\|x\|^2}.
$$
Since $I_n+S$ is a Hermitian matrix $\|I_n+S\|$ is the maximum of the moduli of the eigenvalues of $I_n+S$. Let $x$ in $\mathbb{R}^n$ be an eigenvector of $I_n+S$ corresponding to the eigenvalue $\lambda$ of largest modulus, and let $\overline{x}=(|x_1|,...,|x_n|)$. It follows that:
\begin{equation*}
\| I_n +  S \|  =  | \lambda |
           =  \frac{ | \langle (I_n+S)x, x \rangle |}{\|x\|^2}
          \leq  \frac{ | \langle J_n\overline{x}, \overline{x} \rangle |}{\|x\|^2}
          \leq  \frac{ \|J_n\overline{x}\| \|x\|}{\|x\|^2}
          \leq n.
\end{equation*}
Hence, $\|S\|$ is at most $n-1$.
\end{proof}

\begin{prop} \label{moo}
Let $\mu_S$ denote the least eigenvalue of a Seidel adjacency
matrix $S$, and let $\mathcal{S}$ denote the set of all Seidel
adjacency matrices on $n$ vertices. Then
\begin{enumerate}
\item $\mu:=\mathrm{inf} \{\mu_S : \ S \in \mathcal{S} \} = 1-n,$
\item $I_n+cS$ is a positive semi-definite operator when $c=\frac{1}{n-1}$.
\end{enumerate}
\end{prop}

\begin{proof} By Lemma \ref{largesteigenvalue}, $1-n \leq \mu$. However, $-n$ is the least eigenvalue of $-J_n$ and consequently the Seidel adjacency matrix
$S=I_n-J_n$ has $-n+1$ for a least eigenvalue. Therefore $\mu=-n+1$.

Let $\sigma(S)$ denote the spectrum of $S$. Then
\begin{align*}
\sigma(S) \subseteq [-n+1,n-1] & \Longleftrightarrow \sigma\left( cS \right) \subseteq [-1,1] \\
                               & \Longleftrightarrow \sigma\left(I_n + cS \right) \subseteq [0,2].
\end{align*}
Thus, $I_n+cS$ is a positive semi-definite operator.
\end{proof}

\begin{thm} \label{equality}
Let $X$ be a graph on $n$ vertices and $S$ be the associated Seidel
adjacency matrix.  Then $\displaystyle{e_m^{\infty}(X) \leq 1+\frac{m-1}{n-1}}$.  Furthermore,
$\displaystyle{e_m^{\infty}(X)=1+\frac{m-1}{n-1}}$ if and only if $X$ has an induced subgraph on
$m$ vertices which is complete bipartite or empty.
\end{thm}

Theorem \ref{equality} is a generalization of Bodman and Paulsen's
Theorem 5.3 in \cite{BP}. While our proof is similar, we include it
to provide insight into the relationship between a graph $X$
and the value $e_m^{\infty}(X)$.

\begin{proof} By Lemma \ref{largesteigenvalue} and the triangle inequality the claimed error bound follows:

$$
e_m^{\infty}(X)  \leq  1 + \frac{m-1}{n-1}.
$$

Assume that the graph $X$ has an induced subgraph on $m$ vertices which is complete bipartite or empty. Choose $D$ to
have ones in the places on the diagonal corresponding to the vertices of this subgraph and zeros elsewhere. Then, $D(I_n+Q)D$ is switching equivalent to
$DJ_nD$. Since switching preserves the operator norm, $\| I_m+Q_m\|=m$ implying that $\| Q_m \|=m-1$. Therefore
$$
e_m^{\infty}(X)=\mathrm{max}\{ \| D(I_n+cS)D \|  : D \in \mathcal{D}_m \}=1+\frac{m-1}{n-1}.
$$

Now assume that $e_m^{\infty}(X)=1+\frac{m-1}{n-1}$ or equivalently $\| Q_m \|=m-1$.
Then, for some $D$, $\|D(I_n+Q)D\|=m$.
 Let $x$ be an eigenvector corresponding to the eigenvalue $\pm m$.
Choose a switching matrix $S$ such that all of the entries of $Sx$
 are positive, i.e., $S$ should have $-1's$ on the diagonal in the places where
the entries of $x$ are negative and $1's$ on the other diagonal
entries.
 Using reasoning similar to the proof of Lemma 3.8, all of the entries
of $S(I_n+Q)S$ must be $1's$ in the rows and columns where $D$ has
$1's$ on the diagonal. Otherwise, since $Sx$ has all positive
entries and is the eigenvector corresponding to the largest
eigenvalue of $SD(I_n+Q)DS$, it would be possible to increase the
largest eigenvalue of $SD(I_n+Q)DS=I_m+Q_m^{\prime}$ by flipping signs in $Q_m^{\prime}$, contradicting that the inequality is saturated. Hence, the induced subgraph on these vertices is switching equivalent to the edgeless graph, i.e., this induced subgraph is complete bipartite.
\end{proof}

\begin{cor}\label{infnormcomplete}
Let $X$ be a graph on $n$ vertices.  Then $e_3^{\infty}(X) < 1+ \frac{2}{n-1}$ if
and only if $X$ is switching equivalent to the complete graph,
denoted by $K_n$.
\end{cor}

\begin{cor}\label{infnormempty}
Let $X$ be a graph on $n$ vertices.  Then $e_n^{\infty}(X)=2$ if and
only $X$ is switching equivalent to the empty graph, denoted by
$E_n$.
\end{cor}

The authors used Maple 11, \cite{Map}, to compute infinity norms for arbitrary graphs. These computations were used for the following theorem and example.

\begin{thm} \label{inftyequivalent}
Let $X_1$ and $X_2$ be graphs on $n$ vertices with $n \leq 5$. Then
$X_1$ and $X_2$ are switching equivalent if and only if
$e_m^{\infty}(X_1)=e_m^{\infty}(X_2)$ for $3 \leq m \leq n$.
\end{thm}

\begin{proof}
The cases of $n=1$ and $n=2$ are clear since there is only one switching equivalence class. For $n=3$, recall from Example \ref{3vertexample} that there are exactly two switching equivalent classes. Corollaries \ref{infnormcomplete} and \ref{infnormempty} give that $e_3^{\infty}$ has different values for these two classes.

In the case $n=4$, Example \ref{4vertexample} shows that there are
exactly $3$ switching equivalent classes. Again, using Corollaries
\ref{infnormcomplete} and \ref{infnormempty}, the graphs switching
equivalent to the complete graph and the graph with no edges are
identified. The other graphs are all switching equivalent and are in
the remaining class.

When $n=5$, Corollaries \ref{infnormcomplete} and \ref{infnormempty}
are not enough to identify all of the classes. We still use these
corollaries to distinguish the classes of the complete graph and the
graph with no edges. By explicit computation, the remaining graphs
form five switching equivalent classes. The representatives and
infinity norm values for these classes are given in Table
\ref{inftyequivtable}.

\begin{table}[ht]
\begin{center}
\renewcommand{\arraystretch}{1.5}
\begin{tabular}{|c|c|c|c|c|c|}  \hline
&&&&&\\
Representative &
$$
\xy/r.8pc/:,
{\xypolygon5{~>{}~:{(1.875,0):}{\circ}}},
"1";"2"**@{-},
\endxy
$$
&
$$
\xy/r.8pc/:,
{\xypolygon5{~>{}~:{(1.875,0):}{\circ}}},
"1";"2"**@{-},
"2";"3"**@{-},
\endxy
$$
&
$$
\xy/r.8pc/:,
{\xypolygon5{~>{}~:{(1.875,0):}{\circ}}},
"1";"2"**@{-},
"3";"4"**@{-},
\endxy
$$
&
$$
\xy/r.8pc/:,
{\xypolygon5{~>{}~:{(1.875,0):}{\circ}}},
"1";"2"**@{-},
"2";"3"**@{-},
"3";"4"**@{-},
\endxy
$$
&
$$
\xy/r.8pc/:,
{\xypolygon5{~>{}~:{(1.875,0):}{\circ}}},
"1";"2"**@{-},
"2";"3"**@{-},
"3";"1"**@{-},
\endxy
$$
\\ &&&&& \\ \hline\hline
$e_3^{\infty}(X)$ & $\frac{3}{2}$ & $\frac{3}{2}$ & $\frac{3}{2}$ & $\frac{3}{2}$ & $\frac{3}{2}$ \\ \hline
$e_4^{\infty}(X)$ & $\frac{7}{4}$ & $\frac{7}{4}$ & $1+\frac{\sqrt{5}}{4}$ & $1+\frac{\sqrt{5}}{4}$ & $1+\frac{\sqrt{5}}{4}$ \\ \hline
$e_5^{\infty}(X)$ & $\frac{9}{8}+\frac{\sqrt{33}}{8}$ & $\frac{7}{4}$ & $\frac{9}{8}+\frac{\sqrt{17}}{8}$ & $1+\frac{\sqrt{5}}{4}$ & $\frac{7}{8}+\frac{\sqrt{33}}{8}$ \\ \hline
\end{tabular}
\caption{Infinity norms for graphs on $5$ vertices} \label{inftyequivtable}
\end{center}
\end{table}

\end{proof}

\begin{exmp}\label{Ex:inf6}
This example shows that Theorem \ref{inftyequivalent} fails for $n \geq 6$. The following graphs are not switching equivalent.

$$
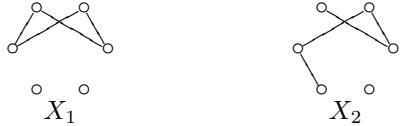
\begin{array}{cc}
\xy/r.8pc/:,
{\xypolygon6{~>{}~:{(1.875,0):}{\circ}}},
"1";"2"**@{-},
"1";"3"**@{-},
"2";"4"**@{-},
"4";"3"**@{-},
\endxy

\hspace{2cm}
&

\xy/r.8pc/:,
{\xypolygon6{~>{}~:{(1.875,0):}{\circ}}},
"1";"2"**@{-},
"1";"3"**@{-},
"2";"4"**@{-},
"4";"5"**@{-},
\endxy
 \\

X_1 \hspace{2cm} & X_2\\
\end{array}
$$
Clearly, $X_1$ is not isomorphic $X_2$, and switching $X_1$ on any subset $\tau$ of $V(X_1)$ will not produce a graph isomorphic to $X_2$. Consequently, these graphs are not switching equivalent. Table \ref{sixsix} shows that $e_m^{\infty}$ does not distinguish the classes of these graphs.

\begin{table}[h]\label{sixsix}
\begin{center}
\renewcommand{\arraystretch}{1.25}
\begin{tabular}{|c|c|c|} \hline
$m$ & $e_m^{\infty}(X_1)$ & $e_m^{\infty}(X_2)$    \\ \hline
3   & $ 1.\bar{4}$  &  $ 1.\bar{4}$  \\ \hline
4   & $ 1.\bar{6}$  &  $ 1.\bar{6}$  \\ \hline
5   & $ 1.\bar{6}$  &  $ 1.\bar{6}$  \\ \hline
6   & $ 1.\bar{6}$  &  $ 1.\bar{6}$  \\ \hline
\end{tabular}
\caption{counter-example} \label{counter-example}
\end{center}
\end{table}
\end{exmp}

\section{Spectrum Determined Switching Equivalent Classes}\label{Sec:DS}

As noted in Section 2, switching equivalent graphs have the same
Seidel spectrum. A natural question is whether the switching
equivalent class of a graph is determined uniquely by its Seidel
spectrum. In \cite{vDH}, van Dam and Haemers survey the known
results and open questions for graphs determined by their spectrum.
While graphs are not determined by their Seidel spectrum, the
analogous question about switching equivalent classes is not so
obvious. The following result has been verified by direct
computation.

\begin{thm}
For $n\leq 7$, the graphs $G$ and $H$ on $n$ vertices are switching
equivalent if and only if their Seidel matrices have the same
spectrum.
\end{thm}

\begin{exmp}\label{Ex:SeiSpec}
No pair of the following three graphs on eight vertices are
switching equivalent, yet their Seidel matrices all have the same
spectrum.

$$
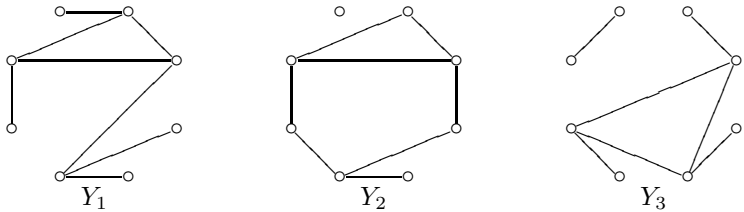
\begin{array}{ccc}
\xy/r1.5pc/:,
{\xypolygon8{~>{}~:{(1.875,0):}{\circ}}},
"1";"2"**@{-},
"1";"4"**@{-},
"1";"6"**@{-},
"2";"3"**@{-},
"2";"4"**@{-},
"4";"5"**@{-},
"6";"7"**@{-},
"6";"8"**@{-},
\endxy

\hspace{1cm}
&

\xy/r1.5pc/:,
{\xypolygon8{~>{}~:{(1.875,0):}{\circ}}},
"1";"2"**@{-},
"1";"4"**@{-},
"1";"8"**@{-},
"2";"4"**@{-},
"4";"5"**@{-},
"5";"6"**@{-},
"6";"7"**@{-},
"6";"8"**@{-},
\endxy
\hspace{1cm}
&

\xy/r1.5pc/:,
{\xypolygon8{~>{}~:{(1.875,0):}{\circ}}},
"1";"2"**@{-},
"1";"5"**@{-},
"1";"7"**@{-},
"3";"4"**@{-},
"5";"6"**@{-},
"5";"7"**@{-},
"7";"8"**@{-},
\endxy
 \\

Y_1 \hspace{1cm} & Y_2 \hspace{1cm} & Y_3 \\
\end{array}
$$
\end{exmp}

In addition to Example \ref{Ex:SeiSpec}, there are six other pairs
of switching equivalent classes on eight vertices, each of which
have the same Seidel spectrum. These examples are the smallest
number of vertices where this occurs. Seidel and others have
evidence of larger examples where nonswitching equivalent graphs
have the same spectrum.

\begin{exmp}\label{26example}
In \cite{BMS}, the authors make reference to four nonequivalent
switching classes on $26$ vertices. This example is expanded in
Section \ref{Sec:26}. The authors of \cite{BP} state that these four
classes have Seidel matrices which are conference matrices, forcing
them to have eigenvalues $\pm 5$, each with multiplicity $13$. This
implies that they are Seidel cospectral.
\end{exmp}

\section{Decks}\label{Sec:Decks}

The notion of the deck of a graph is a commonly used tool for
attempting to determine certain invariants of a graph from its
collection of unlabeled induced subgraphs see \cite{LS}.  While the
deck reconstruction problem has not been solved for graphs in
general, our work suggests that its analogue for switching equivalent classes
will hold.

\begin{defn}
A {\bf vertex-deleted subgraph} of a graph $G$ is a subgraph $G_v$
obtained by deleting a vertex $v$ and its incident edges. The {\bf
deck} of a graph $G$, denoted $\mathcal{D}(G)$, is the family of
unlabeled vertex-deleted subgraphs of $G$; these are called the cards of
the deck.
\end{defn}

\begin{defn}
Let $G$ and $H$ be graphs on the same number of vertices. We say
their decks are isomorphic, denoted $\mathcal{D}(G) \cong
\mathcal{D}(H)$, if there exists a bijection $\pi:\mathcal{D}(G)\to
\mathcal{D}(H)$ such that $\pi(x) \cong x$.  In this case, we call
$H$ a {\bf reconstruction} of $G$. Similarly, we define the notion
of switching equivalent decks, denoted by $\mathcal{D}(G) \sim
\mathcal{D}(H)$, in which case the bijection $\pi$ satisfies $\pi(x)
\in [[x]]$.
\end{defn}

If every reconstruction of $G$ is isomorphic to $G$, we say $G$ is
reconstructible. The reconstruction conjecture as stated in
\cite{LS} is:

\begin{conj}\label{Recon}
Every graph with at least three vertices is reconstructible.
\end{conj}

We call the switching equivalent class $[[H]]$ reconstructible if
$\mathcal{D}(G) \sim \mathcal{D}(H)$ implies that $G\in [[H]]$. This
leads to our switching equivalent reconstruction conjecture.

\begin{conj}\label{SERecon}
Every switching equivalent class on at least 4 vertices is reconstructible.
\end{conj}

A positive result for Conjecture \ref{Recon} would prove Conjecture
\ref{SERecon}. However, these conjectures are not equivalent. Also, many of the classes of graphs for which
Conjecture \ref{Recon} is known, i.e. disconnected graphs, regular
graphs, etc., can not be considered under Conjecture \ref{SERecon}
since their defining properties are not preserved by switching.

Revisiting the counterexamples from the previous sections, we see
that decks differentiate the switching equivalent classes.

\begin{exmp}\label{Deck6}
Consider the graphs from Example \ref{Ex:inf6}. The graphs $X_1$ and
$X_2$ on six vertices are not switching equivalent, and yet
$e_m^\infty(X_1)=e_m^\infty(X_2)$ for $1 \le m \le 6$.
$$
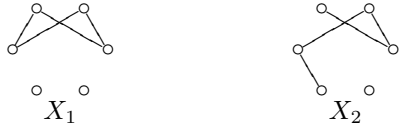
\begin{array}{cc}
\xy/r.8pc/:,
{\xypolygon6{~>{}~:{(1.875,0):}{\circ}}},
"1";"2"**@{-},
"1";"3"**@{-},
"2";"4"**@{-},
"4";"3"**@{-},
\endxy

\hspace{2cm}
&

\xy/r.8pc/:,
{\xypolygon6{~>{}~:{(1.875,0):}{\circ}}},
"1";"2"**@{-},
"1";"3"**@{-},
"2";"4"**@{-},
"4";"5"**@{-},
\endxy
 \\

X_1 \hspace{2cm} & X_2\\
\end{array}
$$
The deck of $X_1$ consists of the graphs
$$
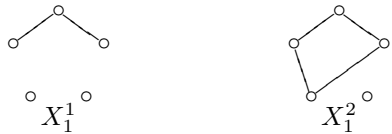
\begin{array}{cc}
\xy/r.8pc/:,
{\xypolygon5{~>{}~:{(1.875,0):}{\circ}}},
"1";"2"**@{-},
"2";"3"**@{-},
\endxy

\hspace{2cm}
&

\xy/r.8pc/:,
{\xypolygon5{~>{}~:{(1.875,0):}{\circ}}},
"1";"2"**@{-},
"1";"4"**@{-},
"2";"3"**@{-},
"4";"3"**@{-},
\endxy
 \\

X_1^1 \hspace{2cm} & X_1^2\\
\end{array}
$$ where $X_1^1$ appears 4 times and $X_1^2$ appears twice. Switching $X_1^1$ on
 its vertices of even degree gives $X_1^2$, which is an Euler graph. The deck of $X_2$ consists of the graphs
$$
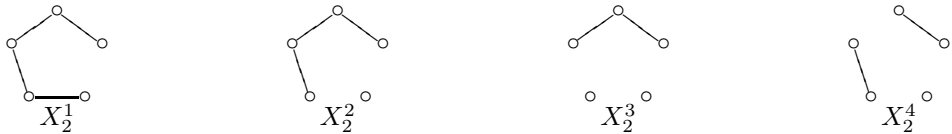
\begin{array}{cccc}
\xy/r.8pc/:,
{\xypolygon5{~>{}~:{(1.875,0):}{\circ}}},
"1";"2"**@{-},
"2";"3"**@{-},
"3";"4"**@{-},
"4";"5"**@{-},
\endxy

\hspace{2cm}
&
\xy/r.8pc/:,
{\xypolygon5{~>{}~:{(1.875,0):}{\circ}}},
"1";"2"**@{-},
"2";"3"**@{-},
"3";"4"**@{-},
\endxy

\hspace{2cm}
&
\xy/r.8pc/:,
{\xypolygon5{~>{}~:{(1.875,0):}{\circ}}},
"1";"2"**@{-},
"2";"3"**@{-},
\endxy

\hspace{2cm}
&

\xy/r.8pc/:,
{\xypolygon5{~>{}~:{(1.875,0):}{\circ}}},
"1";"2"**@{-},
"4";"3"**@{-},
\endxy
 \\

X_2^1 \hspace{2cm} & X_2^2 \hspace{2cm} & X_2^3 \hspace{2cm} & X_2^4 \\
\end{array}
$$
where $X_2^1$ and $X_2^4$ appear once, and $X_2^2$ and $X_2^3$ each appear twice. Switching $X_2^3$ on its even degree vertices gives a graph isomorphic to $X_1^2$, but none of the other graphs in this deck are switching equivalent to $X_1^2$ by Corollary \ref{EulerIso}. Therefore, $\mathcal{D}(X_1)\nsim\mathcal{D}(X_2)$.

\end{exmp}

\begin{exmp}\label{Deck8}
Consider the three graphs from Example \ref{Ex:SeiSpec}. For ease of reading, the following decks have already been switched to their Euler graph representatives using Theorem \ref{EulerSwitch}. The deck of $Y_1$ contains the following graphs
$$
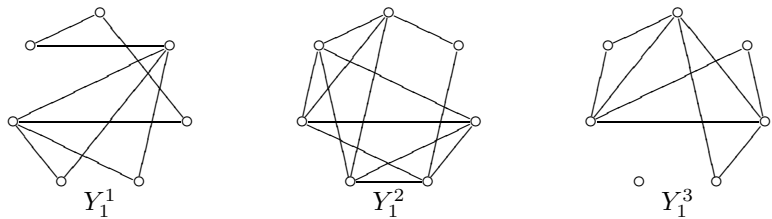
\begin{array}{ccc}
\xy/r1.5pc/:,
{\xypolygon7{~>{}~:{(1.875,0):}{\circ}}},
"1";"3"**@{-},
"1";"4"**@{-},
"1";"5"**@{-},
"1";"6"**@{-},
"2";"3"**@{-},
"2";"7"**@{-},
"4";"5"**@{-},
"4";"6"**@{-},
"4";"7"**@{-},
\endxy

\hspace{1cm}
&

\xy/r1.5pc/:,
{\xypolygon7{~>{}~:{(1.875,0):}{\circ}}},
"1";"2"**@{-},
"1";"6"**@{-},
"2";"3"**@{-},
"2";"4"**@{-},
"2";"5"**@{-},
"3";"4"**@{-},
"3";"5"**@{-},
"3";"7"**@{-},
"4";"6"**@{-},
"4";"7"**@{-},
"5";"6"**@{-},
"5";"7"**@{-},
"6";"7"**@{-},
\endxy
\hspace{1cm}
&

\xy/r1.5pc/:,
{\xypolygon7{~>{}~:{(1.875,0):}{\circ}}},
"1";"4"**@{-},
"1";"7"**@{-},
"2";"3"**@{-},
"2";"4"**@{-},
"2";"6"**@{-},
"2";"7"**@{-},
"3";"4"**@{-},
"4";"7"**@{-},
"6";"7"**@{-},
\endxy
 \\

Y_1^1 \hspace{1cm} & Y_1^2 \hspace{1cm} & Y_1^3 \\
\end{array}
$$ with $Y_1^1$ and $Y_1^2$ each occurring 3 times and $Y_1^3$ occurring twice.
The deck of $Y_2$ contains the following graphs
$$
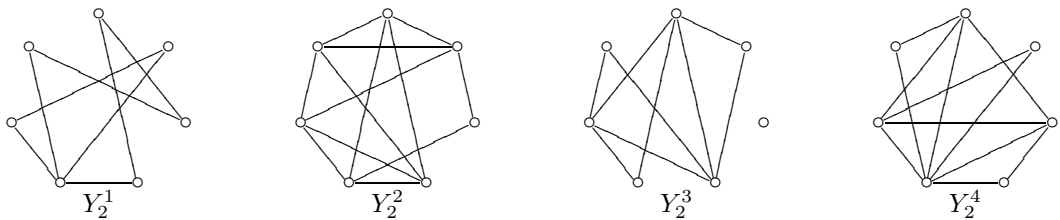
\begin{array}{cccc}
\xy/r1.5pc/:,
{\xypolygon7{~>{}~:{(1.875,0):}{\circ}}},
"1";"4"**@{-},
"1";"5"**@{-},
"2";"6"**@{-},
"2";"7"**@{-},
"3";"5"**@{-},
"3";"7"**@{-},
"4";"5"**@{-},
"5";"6"**@{-},
\endxy

\hspace{1cm}
&

\xy/r1.5pc/:,
{\xypolygon7{~>{}~:{(1.875,0):}{\circ}}},
"1";"2"**@{-},
"1";"3"**@{-},
"1";"4"**@{-},
"1";"7"**@{-},
"2";"3"**@{-},
"2";"5"**@{-},
"2";"6"**@{-},
"3";"4"**@{-},
"3";"6"**@{-},
"4";"5"**@{-},
"4";"6"**@{-},
"5";"6"**@{-},
"5";"7"**@{-},
\endxy
\hspace{1cm}
&

\xy/r1.5pc/:,
{\xypolygon7{~>{}~:{(1.875,0):}{\circ}}},
"1";"2"**@{-},
"1";"6"**@{-},
"2";"4"**@{-},
"2";"5"**@{-},
"2";"6"**@{-},
"3";"4"**@{-},
"3";"6"**@{-},
"4";"5"**@{-},
"4";"6"**@{-},
\endxy
\hspace{1cm}
&

\xy/r1.5pc/:,
{\xypolygon7{~>{}~:{(1.875,0):}{\circ}}},
"1";"4"**@{-},
"1";"5"**@{-},
"2";"3"**@{-},
"2";"4"**@{-},
"2";"5"**@{-},
"2";"7"**@{-},
"3";"5"**@{-},
"4";"5"**@{-},
"4";"7"**@{-},
"5";"6"**@{-},
"5";"7"**@{-},
"6";"7"**@{-},
\endxy

 \\

Y_2^1 \hspace{1cm} & Y_2^2 \hspace{1cm} & Y_2^3 \hspace{1cm} & Y_2^4 \\
\end{array}
$$ with $Y_2^1$ and $Y_2^2$ each occurring 3 times and $Y_2^3$ and $Y_2^4$ each occurring once.
The deck of $Y_3$ contains the following graphs
$$
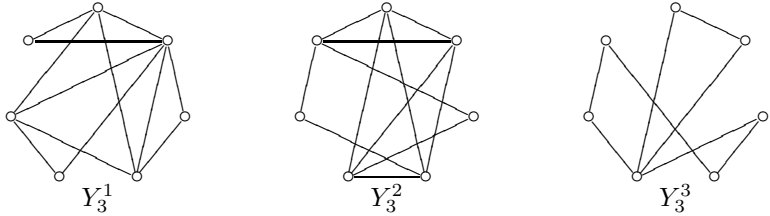
\begin{array}{ccc}
\xy/r1.5pc/:,
{\xypolygon7{~>{}~:{(1.875,0):}{\circ}}},
"1";"2"**@{-},
"1";"3"**@{-},
"1";"4"**@{-},
"1";"5"**@{-},
"1";"6"**@{-},
"1";"7"**@{-},
"2";"3"**@{-},
"2";"4"**@{-},
"2";"6"**@{-},
"4";"5"**@{-},
"4";"6"**@{-},
"6";"7"**@{-},
\endxy

\hspace{1cm}
&

\xy/r1.5pc/:,
{\xypolygon7{~>{}~:{(1.875,0):}{\circ}}},
"1";"2"**@{-},
"1";"3"**@{-},
"1";"5"**@{-},
"1";"6"**@{-},
"2";"3"**@{-},
"2";"5"**@{-},
"2";"6"**@{-},
"3";"4"**@{-},
"3";"7"**@{-},
"4";"6"**@{-},
"5";"6"**@{-},
"5";"7"**@{-},
\endxy
\hspace{1cm}
&

\xy/r1.5pc/:,
{\xypolygon7{~>{}~:{(1.875,0):}{\circ}}},
"1";"2"**@{-},
"1";"5"**@{-},
"2";"5"**@{-},
"3";"4"**@{-},
"3";"6"**@{-},
"4";"5"**@{-},
"5";"7"**@{-},
"6";"7"**@{-},
\endxy
 \\

Y_3^1 \hspace{1cm} & Y_3^2 \hspace{1cm} & Y_3^3 \\
\end{array}
$$ with $Y_3^1$ occurring twice and $Y_3^2$ and $Y_3^3$ each occurring three times. By quick inspection and Corollary \ref{EulerIso}, no pair of these three decks is switching equivalent.
\end{exmp}

Examples \ref{Deck6} and \ref{Deck8} show that decks differentiate switching equivalent classes in cases where the infinity norm and Seidel spectrum do not. Using programs
written in GAP, \cite{GAP}, representatives for the switching
equivalent classes have been constructed for $4\leq n\leq 10$
vertices. Additional programs in GAP have verified Conjecture
\ref{SERecon} for these representatives.

Conjecture \ref{SERecon} can be rewritten as a test of switching equivalence as follows.

\begin{conj}\label{deckequivalent}
Let $X_1$ and $X_2$ be graphs on $n$ vertices. Then $X_1$ and $X_2$
are switching equivalent if and only if $\mathcal{D}(G) \sim
\mathcal{D}(H)$.
\end{conj}

\section{One Norms}\label{Sec:OneNorms}

While decks seem to differentiate switching equivalent classes of graphs, there is another candidate which is more strongly tied to our motivation, as described in Appendix \ref{App:Frames}.

\begin{defn} \label{onenormdefA}
Let $\mathcal{D}_m$ denote the set of diagonal matrices that have
exactly $m$ diagonal entries equal to one and $n-m$ entries equal to
zero.  Given a graph $X$ on $n$ vertices, set
$$
e_m^1(X):= {n \choose m}^{-1} \sum_{D \in \mathcal{D}_m} \| D(I+cS)D \|,
$$
where $S$ is the Seidel adjacency matrix of $X$, $c=\frac{1}{n-1}$, and the norm of the matrix is understood to be the operator norm.
\end{defn}

The $1$-norm is related to the infinity norm defined in Section \ref{Sec:InfNorms} since it is an average of the same list of numbers of which the infinity norm was returning the maximum. Another way to think about the $1$-norm is to consider all induced subgraphs of $X$ on $m$ vertices. This collection of subgraphs can be partitioned according to their infinity norms. The computation of $e_m^1(X)$ follows from counting the number of graphs in each element of the partition.

\begin{exmp}
Returning to Example \ref{Ex:inf6}, Table \ref{counter-example} is expanded to Table \ref{counter-exampleB}, giving the $1$-norm values.
\begin{table}[h]
\begin{center}
\renewcommand{\arraystretch}{1.25}
\begin{tabular}{|c|c|c|c|c|} \hline
$m$ & $e_m^{\infty}(X_1)$ & $e_m^{\infty}(X_2)$ &   $e_m^1(X_1)$   &   $e_m^1(X_2)$    \\ \hline
3   & $ 1.\bar{4}$  &  $ 1.\bar{4}$ &   $ 1.32$ &   $ 1.32$   \\ \hline
4   & $ 1.\bar{6}$  &  $ 1.\bar{6}$ &   $ 1.479$ &   $ 1.442$   \\ \hline
5   & $ 1.\bar{6}$  &  $ 1.\bar{6}$ &   $ 1.\bar{6}$ &   $ 1.52$    \\ \hline
6   & $ 1.\bar{6}$  &  $ 1.\bar{6}$ &   $ 1.\bar{6}$ &   $ 1.\bar{6}$    \\ \hline
\end{tabular}
\caption{counter-example} \label{counter-exampleB}
\end{center}
\end{table}
In this case, where the infinity norm failed, the $1$-norm differentiates these classes of graphs.
\end{exmp}

As for the infinity norm, the authors implemented programs in Maple 11, \cite{Map}, to compute $1$-norms of arbitrary graphs. Using the class representatives computed in GAP, \cite{GAP}, the $1$-norms have been calculated for all switching equivalent classes on $4\leq n\leq 10$ vertices. The obtained results support the following conjecture.

\begin{conj}\label{oneequivalent}
Let $X_1$ and $X_2$ be graphs on $n$ vertices. Then, $X_1$ and $X_2$ are switching equivalent if and only if
$e_m^{1}(X_1)=e_m^{1}(X_2)$ for $1 \leq m \leq n$.
\end{conj}

\section{An Important Example}\label{Sec:26}

Since the $4$ nonswitching equivalent classes on $26$ vertices mentioned in Example \ref{26example} are well known and are a clear counterexample in the Seidel spectrum case, we give the results of our conjectures applied to them here. We are grateful to Spence for providing representatives for these classes in \cite{Spe}. For the purposes of this section, we refer to these four representatives as $Q_1$, $Q_2$, $Q_3$, and $Q_4$.

\subsection{Decks}
Using programs written in GAP, \cite{GAP}, the decks of these graphs are quickly produced. To simplify checking switching equivalence for graphs on $25$ vertices, all of the cards are switched to their unique Euler representative as described in Theorem \ref{EulerSwitch}. This gives four sets of $26$ $12$-regular graphs on $25$ vertices. Using the GAP package GRAPE, \cite{GRAPE}, which relies on the C-program nauty, \cite{Nauty}, the isomorphism classes of these Euler graphs have been identified. By Corollary \ref{EulerIso}, the isomorphism classes of the Euler representatives give the switching equivalent classes of the cards in the deck. For two decks to be switching equivalent, there have to be the same number of cards in each switching equivalent class. The graphs $Q_1$, $Q_2$, $Q_3$, and $Q_4$ have $8$, $1$, $2$, and $4$ switching equivalent classes represented in their decks, respectively. Therefore, Conjecture \ref{deckequivalent} holds for these four important, see \cite{BP}, switching equivalent classes on $26$ vertices.

\subsection{One Norms}

The Interlacing Theorem gives evidence that the matrices $Q_1$,
$Q_2$, $Q_3$, and $Q_4$ are a good test for Conjecture
\ref{oneequivalent}.

\begin{thm}[The Interlacing Theorem]\label{interthm}
Let $A$ be an $n\times n$ symmetric matrix with eigenvalues
$$\lambda_1\geq\lambda_2\geq\dots\geq\lambda_n,$$
and let $B$ be obtained by removing the $i^{th}$ row and column of $A$ and suppose $B$ has eigenvalues
$$\mu_1\geq\mu_2\geq\dots\geq\mu_{n-1}.$$
Then the eigenvalues of $B$ interlace those of $A$, that is,
$$\lambda_1\geq\mu_1\geq\lambda_2\geq\dots\geq\lambda_{n-1}\geq\mu_{n-1}\geq\lambda_n.$$
\end{thm}

A proof of Theorem \ref{interthm} can be found in \cite{LS}. Since
$Q_1$, $Q_2$, $Q_3$, and $Q_4$ all have eigenvalues $5$ and $-5$,
each with multiplicity $13$, Theorem \ref{interthm} allows us to
find the values of $e_m^1(Q_i)$ without direct computation for $14
\le m \le 26$. Recall from Definition \ref{onenormdefA}, $1$-norms
are averaged sums of $\|D(I+cX)D\|$, where $X$ is the Seidel
adjacency matrix, $I$ is an identity matrix, and $D$ is a matrix
which deletes $n-m$ rows and their corresponding columns. Applying
the Interlacing Theorem, we get that when $n-m<13$,
$\|D(I+\frac{1}{25}Q_i)D\|=1.2$ for $1\leq i\leq 4$. So,
$e_m^1(Q_i)=1.2$ for $14\leq m\leq 26$. This limited variation
provides a good test for Conjecture \ref{oneequivalent}.

We used our programs written in Maple 11, see \cite{Map}, to
evaluate the $1$-norms for these four classes. The results
summarized in Table \ref{26vert} show that Conjecture
\ref{oneequivalent} holds for these four important switching
equivalent classes on $26$ vertices.

\begin{table}[h]
\begin{center}
\renewcommand{\arraystretch}{1.25}
\begin{tabular}{|c|c|c|c|c|} \hline
$m$ & $e_m^1(Q_1)$ & $e_m^1(Q_2)$ &   $e_m^1(Q_3)$   &   $e_m^1(Q_4)$    \\ \hline
$3$   & $1.06$ & $1.06$ & $1.06$ & $1.06$   \\ \hline
$4$   & $1.0873899540482$ & $1.0873899540482$ & $1.0873899540482$ & $1.0873899540482$ \\ \hline
$5$   & $1.1071147791905$ & $1.1071399835086$ & $1.1069232263711$ & $1.1069433898254$ \\ \hline
$6$   & $1.1253569536629$ & $1.1253899928320$ & $1.1251058556967$ & $1.1251322799265$ \\ \hline
\end{tabular}
\caption{26 Vertex Example} \label{26vert}
\end{center}
\end{table}

In light of our results, we
feel that Conjectures \ref{oneequivalent} and \ref{deckequivalent} deserve further study.

\appendix
\section{Motivation and Frame Theory}\label{App:Frames}

In this section we give a brief introduction to frame theory in
order to discuss the motivation behind studying the qualities
$e_m^{\infty}(X)$ and $e_m^{p}(X)$ for a given graph $X$. Strohmer
and Heath in \cite{SH} first introduced the frame theory community
to results in graph theory which yield examples of $2$-uniform
frames. In \cite{HP} and \cite{BP}, Bodmann, Holmes and Paulsen take
advantage of the one-to-one correspondence between regular
two-graphs and $2$-uniform frames to give a complete list of all
$2$-uniform $(n,k)$-frames for $n \leq 50$.

\begin{defn}
Let $\mathcal{H}$ be a Hilbert space, real or complex, and let $F = \{ f_i \}_{i\in\mathbb{I}} \subset \mathcal{H}$ be a subset. Then, $F$ is a {\bf frame} for $\mathcal{H}$ provided that there are two constants $C, \ D > 0$ such that the norm inequalities
$$
C \cdot \| x \|^2 \leq \sum_{j \in \mathbb{I}} | \langle x,f_j \rangle |^2 \leq D \cdot \|x \|^2
$$
hold for every $x \in \mathcal{H}$. Here $\langle \cdot , \cdot \rangle$ denotes the inner product of two vectors which is by convention conjugate linear in the second entry if $\mathcal{H}$ is a complex Hilbert space.

When $C=D=1$, then  $F$ is called a {\bf Parseval frame}. A frame is called {\bf uniform} provided there is a constant $c$ so that $\|f \|=c$ for all $f \in F$.

The linear map $V: \mathcal{H} \rightarrow l_2(\mathbb{I})$ defined by
$$
(Vx)_i=\langle x, f_i \rangle
$$
is called the {\bf analysis operator}. When $F$ is a Parseval frame, then $V$ is an isometry; and its adjoint, $V^*$, acts as a left inverse of $V$.
\end{defn}

For the purposes of this paper, $\mathcal{H}$ will be a finite dimensional real Hilbert space and frames for these spaces will consist of finitely many vectors. If the dimension of $\mathcal{H}$ is $k$, then we will identify $\mathcal{H}$ with $\mathbb{R}^k$.

\begin{defn}
Let $\mathcal{F}(n,k)$ denote the collection of all Parseval frames for $\mathbb{R}^k$ consisting of $n$ vectors and refer to such a frame as a real {\bf (n,k)-frame}. Thus, a uniform $(n,k)$-frame is a uniform Parseval frame for $\mathbb{R}^k$ with $n$ vectors.
\end{defn}

The idea behind treating frames as codes is studied in depth in \cite{HP} and \cite{BP}. For a more detailed study of uniform $(n,k)$-frames and $2$-uniform $(n,k)$-frames see \cite{BP}, and for an excellent survey on frames see \cite{Cas} or \cite{KC}. Given a vector $x$ in $\mathbb{R}^k$ and an $(n,k)$-frame with analysis operator $V$, consider the vector
$Vx$ in $\mathbb{R}^n$ as an encoded version of $x$, and simply decode $Vx$ by applying $V^*$. Let $E$ denote the diagonal matrix of $m$ zeros and $n-m$ ones. Thus the vector $EVx$ is just the vector $Vx$ with $m$-components erased corresponding to the zeros in the diagonal entries of $E$.
One way to decode the received vector $EVx$ with $m$ erasures is to again apply $V^*$. The error in reconstructing $x$ this way is given by
$$
\| x-V^*EV \| = \| V^*(I-E)Vx \| = \| V^*DVx \|
$$
where $D$ is the diagonal matrix of $m$ ones and $n-m$ zeros. This is only one of several methods possible for reconstructing $x$. However, it is this particular method which led Bodmann and Paulsen in \cite{BP} to introduce the following definition. The first quantity in Definition \ref{inftynormdef} represents the maximal norm of an error operator given that some set of $m$ erasures occurs, and the second quantity represents an $l^p$-average of the norm of the error operator over the set of all possible $m$ erasures.

\begin{defn} \label{inftynormdef}
Let $\mathcal{D}_m$ denote the set of diagonal matrices that have
exactly $m$ diagonal entries equal to one and $n-m$ entries equal to
zero.  Given an $(n,k)$-frame $F$, set
$$
e_m^{\infty}(F):=\mathrm{max}\{ \| V^*DV \|  : D \in \mathcal{D}_m \},
$$
and for $1 \leq p$,
$$
e_m^p(F)=\left\{ {n \choose m}^{-1} \sum_{D \in \mathcal{D}_m} \| V^*DV \|^p \right\}^{\frac{1}{p}},
$$
where $V$ is the analysis operator of $F$, and the norm of the matrix is understood to be the operator norm.
\end{defn}

\begin{defn}
$F$ is called a {\bf 2-uniform (n,k)-frame} provided that $F$ is a uniform $(n,k)$-frame, and in addition
$\| V^*DV \|$ is a constant for all $D$ in $\mathcal{D}_2$.
\end{defn}

Theorem \ref{framegraphthm} below is a restatement of Theorems 4.7 and 4.8 from \cite{BP}. It states the one-to-one correspondence between regular two-graphs and $2$-uniform frames used to give a complete list of all pairs
$(n,k)$ for $n \leq 50$ for which $2$-uniform $(n,k)$ frames exist over the reals, together with what is known about the numbers of frame equivalence classes. Unlike uniform frames, $2$-uniform frames do not exist for all values of $k$ and $n$. However, $2$-uniform frames turn out to be optimal for one and two erasures when they do exist. Note that a complete list of all $2$-uniform frames over the complex field for $n \leq 50$ is still not known.

\begin{thm} \label{framegraphthm}
The following are equivalent:
\begin{enumerate}
\item $Q$ is an $n \times n$ signature matrix of a real
$2$-uniform $(n,k)$-frame.
\item $Q$ is the Seidel adjacency matrix of a graph on $n$ vertices with $2$ eigenvalues and in this case, $k$ is the multiplicity
of the largest eigenvalue.
\item  $Q$ is the Seidel
adjacency matrix of a graph on $n$ vertices whose switching class is
a regular two-graph on $n$ vertices with parameter $\alpha$.
\end{enumerate}
\end{thm}

In \cite{HP}, given a $2$-uniform $(n,k)$-frame $F$ with analysis operator $V$,
Holmes and Paulsen show that the projection $P=VV^{*}$ can be
written as
\begin{equation} \label{cpq}
P=\frac{k}{n}I+c_{n,k}Q
\end{equation}
where $Q$ satisfies the conditions $q_{ii}=0$ and $|q_{ij}|=1$
for $i \not= j$.  Furthermore, $Q$ has precisely two eigenvalues and
\begin{equation}
c_{n,k} = \sqrt{\frac{k(n-k)}{n^2(n-1)}}. \label{constant}
\end{equation}
The projection matrix $P$ is called the \textit{autocorrelation matrix}
of $F$, and the $n\times n$ self-adjoint matrix $Q$ is called the
\textit{signature matrix} of $F$.  The rank of the projection $P$ is $k$ where the eigenvalues $0$ and $1$ have multiplicities $n-k$ and $k$ respectively. Note that the constant $c_{n,k}$ in equation (\ref{constant}) is such that $c_{n,k}=\| V^*DV \|$ for all $D$ in $\mathcal{D}_2$.

We are now ready to discuss the motivation for the definition of $e_m^{\infty}(X)$ and $e_m^1(X)$ stated in Sections \ref{Sec:InfNorms} and \ref{Sec:OneNorms}. Consider an arbitrary 2-uniform
$(n,k)$-frame $F$, the
maximal norm of the error operator given that some set of $m$
erasures occurs, is given by the formulas,
\begin{align*}
e_m^{\infty}(F)& = \mathrm{max}\{ \| V^*DV \|  : D \in \mathcal{D}_m \}\\
               & = \mathrm{max}\{ \| DVV^*D \|  : D \in \mathcal{D}_m \}\\
               & = \mathrm{max}\{ \| D(\frac{k}{n}I+c_{n,k}Q)D \|  : D \in \mathcal{D}_m \}
\text{ and } \\
e_m^{\infty}(F) & \leq \frac{k}{n} + c_{n,k}(m-1),
\end{align*}
with equality if and only if the corresponding graph $X_F$ contains
an induced subgraph on $m$ vertices that is complete bipartite or
empty (Theorem 5.3 in \cite{BP}).

If $Q$ is any signature
matrix
$$
\frac{k}{n}I +c_{n,k}Q=\frac{k}{n}(I+\frac{1}{| \lambda_1 |}Q)
$$
where the constant $\lambda_1=-\sqrt{\frac{k(n-1)}{n-k}}$ is the least eigenvalue of $Q$, then $I+\frac{1}{| \lambda_1| }Q$ is a positive operator.  Thus, computing $\|V^*DV\|$ is equivalent to computing the largest eigenvalue of $V^*DV$.
However, given an arbitrary Seidel
adjacency matrix, $S$, it no longer makes sense to introduce $k$ and $c_{n,k}$. This is because
the operator $\frac{k}{n}I +c_{n,k}S$ need not be a projection, and more importantly, $I+\frac{1}{|\lambda_1|}S$ need not be a positive
operator.

However, by Proposition \ref{moo}, there is a constant $c$ such that $I+\frac{1}{ c }S$ is a positive operator for any Seidel adjacency matrix $S$, and we can compute
$\|D(I+\frac{1}{ c }S)D\|$ by finding the largest eigenvalue of $D(I+\frac{1}{ c }S)D$. While $c_{n,k}$ is sufficient for making $I_n+c_{n,k}Q$ a positive operator for any signature matrix $Q$, it is not sufficient for making $I_n+c_{n,k}S$ a positive operator for any Seidel matrix $S$.
This is what we mean when we say that Definition \ref{inftynormdefA} is a generalization of Definition \ref{inftynormdef}.

\end{document}